\documentclass[11pt]{amsart}

\usepackage{amssymb}
\input epsf

\numberwithin{equation}{section}

\newtheorem{lem}[subsubsection]{Lemma}
\newtheorem{prop}[subsubsection]{Proposition}

\newtheorem{conj}[subsubsection]{Conjecture}
\newtheorem{thm}[subsubsection]{Theorem}

\theoremstyle{remark}
\newtheorem{rem}[subsubsection]{Remark}
\newtheorem{ex}[subsubsection]{Example}

\newcommand{\iso}{\buildrel{\sim}\over{\longrightarrow}}

\def\im{\mathop{\rm Im}\nolimits \,}

\def\id{{\rm id}}

\newcommand{\cA}{{\mathcal A}}
\newcommand{\cB}{{\mathcal B}}
\newcommand{\cC}{{\mathcal C}}

\newcommand{\cL}{{\mathcal L}}

\newcommand{\cO}{{\mathcal{O}}}

\newcommand{\fa}{{\mathfrak a}}

\newcommand{\epi}{\twoheadrightarrow}
\newcommand{\mono}{\hookrightarrow}

\usepackage[usenames]{color}

\newcommand{\BA}{{\mathbb{A}}}

\newcommand{\BG}{{\mathbb{G}}}

\newcommand{\BN}{{\mathbb{N}}}

\newcommand{\BZ}{{\mathbb{Z}}}

\DeclareMathOperator{\Fib}{{Fib}}
\DeclareMathOperator{\Ker}{{Ker}}
\DeclareMathOperator{\Lie}{{Lie}}
\DeclareMathOperator{\Spec}{{Spec}}

\DeclareMathOperator{\Mor}{{Mor}}
\DeclareMathOperator{\Maps}{{Maps}}
\DeclareMathOperator{\oMaps}{{\Maps^{\circ}}}
\DeclareMathOperator{\Sym}{{Sym}}

\newcommand{\cY}{\mathcal{Y}}

\begin{document}

\thanks{Partially supported by NSF grant DMS-1303100}

\title[Grinberg -- Kazhdan theorem and Newton groupoids]{The Grinberg -- Kazhdan formal arc theorem and the Newton groupoids}

\author{Vladimir Drinfeld}

\address{Dept. of Math., Univ. of Chicago, 5734 University 
Ave., Chicago, IL 60637}


\begin{abstract}
We first prove the Grinberg - Kazhdan formal arc theorem without any assumptions on the characteristic. This part of the article is equivalent to arXiv:math-AG/0203263.

Then we try to clarify the geometric ideas behind the proof by introducing the notion of Newton groupoid (which is related to Newton's method for finding roots).  Newton groupoids are certain groupoids in the category of schemes associated to any generically etale morphism from a locally complete intersection to a smooth variety.
\end{abstract}


\maketitle

\section{Introduction}
\subsection{Subject of the article}

In \S\ref{s:GK} of this article (which is equivalent to the preprint \cite{Dr}) we formulate and prove the Grinberg -- Kazhdan theorem \cite{GK} without any assumptions on the characteristic of the field. The proof is based on the ideas that go back to the 17th century 
(namely, the implicit function theorem or equivalently, Newton's method for finding roots) and the 19th century (the Weierstrass division theorem).

The goal of the rest of the article is to clarify the geometric ideas behind the proof from \S\ref{s:GK}. In particular, we introduce \emph{Newton groupoids}.
These are certain groupoids in the category of schemes, which are related to Newton's method for finding roots.  
They are associated to any generically etale morphism from a locally complete intersection to a smooth variety (both schemes are assumed separated).
Let us note that Newton groupoids are used already in \S\ref{s:GK} behind the scenes.

The main message of \S\ref{s:Intro Newton groupoid}-\ref{s:Defining Newton} is that the notion of groupoid in the category of schemes is useful ``in real life" and that Lie algebroids provide intuition that helps us understand groupoids.
We also demonstrate that it is easy to construct smooth groupoids acting on a given variety in a generically transitive way. (This is in contrast with the situation for algebraic group actions).

\subsection{Structure of the article}
In \S\ref{s:GK} we formulate and prove the Grinberg -- Kazhdan theorem.
In \S\ref{s:rephrasing} we slightly rephrase the proof.

In  \S\ref{s:Intro Newton groupoid} we first recall basic facts about groupoids and stacks. Then we formulate the main properties of the Newton groupoids (see \S\ref{sss:thegoal}, \S\ref{sss:LieGamma_r} and \S\ref{ss:r and r+1}-\ref{ss:the group scheme}.).

In \S\ref{s:Defining Newton} we define the Newton groupoids and verify their properties. The key formula is \eqref{e:tilde x_i}.

\section{The Grinberg -- Kazhdan theorem}  \label{s:GK}
\subsection{Formulation of the theorem}
Let $X$ be a scheme of finite type over a field $k$, and 
$X^\circ \subset X$ the smooth part of $X$.  Consider the scheme $\cL(X)$
of formal arcs in $X$.  The $k$-points of $\cL(X)$ are just maps
${\rm Spec} \, k[[t]] \to X$. Let $\cL^\circ(X)$ be the
open subscheme of arcs whose image is not contained in
$X \setminus X^\circ$.  Fix an arc $\gamma_0:{\rm Spec}\, k[[t]]\to X$ in
$\cL^\circ(X)$, and let $\cL(X)_{\gamma_0}$ be the formal
neighborhood of $\gamma_0$ in $\cL(X)$. We will give a simple proof of
the following theorem, which was proved by M.~Grinberg and D.~Kazhdan
\cite{GK} for fields $k$ of characteristic $0$.

\begin{thm}\label{main}
Suppose that $\gamma_0(0)$ is not an isolated point\footnote{If $\gamma_0(0)$ is an isolated point of $X$ then $\gamma_0$ is an isolated point of $\cL^\circ(X)$, so $\cL(X)_{\gamma_0}=\Spec k$.} of $X$. Then there exists a scheme $Y = Y(\gamma_0)$ of finite type over $k$,
and a point $y \in Y(k)$, such that
$\cL(X)_{\gamma_0}$ is isomorphic to $D^{\infty}\times Y_y$
where $Y_y$ is the formal neighborhood of $y$ in $Y$ and 
$D^{\infty}$ is the product of countably many copies of the formal disk
$D:={\rm Spf}\, k[[t]]$. \end{thm}

The proof will be given in \S\ref{ss:proofGK}; its key idea is formulated at the end of the second paragraph of \S\ref{ss:proofGK}.

\medskip

\noindent

\subsection{Convention} Throughout this paper, a {\it test-ring} $A$ is a
local commutative unital $k$-algebra with residue field $k$ whose
maximal ideal $m$ is nilpotent. If $S$ is a scheme over $k$ and 
$s\in S(k)$ is a $k$-point we think of the formal neighborhood $S_s$ in
terms of its functor of points $A \mapsto S_s (A)$ from test-rings to
sets. For instance, $A$-points of $\cL(X)_{\gamma_0}$ are $A[[t]]$-points
of $X$ whose reduction modulo $m$ equals $\gamma_0$.

\subsection{Warning} \label{ss:warning}
\begin{rem}
$\cL(X)_{\gamma_0}=\cL(X_1)_{\gamma_0}$ where $X_1\subset X$ is the
closure of the connected component of $X^\circ$ containing 
$\gamma_0({\rm Spec} \, k((t))\,)$. So we can assume that $X$ is reduced
and irreducible. But $Y$ is, in general, neither reduced nor irreducible
(e.g., see the following example).
\end{rem}

\begin{ex}    \label{ex:xy+something}
Let $X$ be the hypersurface $yx_{n+1}+g(x_1,\ldots,x_n)=0$, where $g$ is a polynomial vanishing at $0$.  Let $\gamma_0(t)$
be defined by $$x_{n+1}^0(t)=t,y^0(t)=x_1^0(t)=\ldots=x_n^0(t)=0.$$ Then
one can define $Y$ to be the hypersurface $g(x_1,\ldots,x_n)=0$ and $y$
to be the point~$0\in Y$. Indeed, by the Weierstrass division theorem
(a.k.a preparatory lemma) for any test-ring $A$ every $A$-deformation of 
$x_{n+1}^0(t)=t$ can be uniquely written as $$x_{n+1}(t)=(t-\alpha)u(t),$$ where $\alpha$
belongs to the maximal ideal $m\subset A$ and $u\in 1+m[[t]]$. 
Given $\alpha$, $u$, and $x_1(t)\ldots,x_n(t)\in m[[t]]$,
there is at most one $y(t)\in m[[t]]$ such that
$$y(t)x_{n+1}(t)+g(x_1(t),\ldots,x_n(t))=0,$$ and $y(t)$ exists if and
only if $g(x_1(\alpha ),\ldots,x_n(\alpha ))=0$. Writing $x_1(t),\ldots , x_n(t)$ as
\[
x_i(t)=\xi_i+(t-\alpha )\tilde x_i (t), \quad \xi_i\in m, \tilde x_i\in m[[t]],
\]
we see that the set of $A$-points of $\cL(X)_{\gamma_0}$ identifies with the set of collections
\begin{equation}   \label{e:collections}
(\alpha, u(t),\tilde x_1(t),\ldots \tilde x_n(t), \xi_1,\ldots \xi_n), 
\end{equation}
where $\xi_1,\ldots, \xi_n\in m$ satisfy the equation
\begin{equation}  \label{e:the-equation}
g(\xi_1,\ldots,\xi_n ))=0
\end{equation}
and $\alpha\in m$, $u\in 1+m[[t]]$, $\tilde x_i\in m[[t]]$ are ``free variables".
\end{ex}

\subsection{Proof of Theorem \ref{main}.} \label{ss:proofGK}
We can assume that $X$ is a closed subscheme of an affine space. Then
there is a closed subscheme $X'$ of the affine space such that
$X'\supset X$, $X'$ is a complete intersection, and the image of
our arc $\gamma_0$ is not contained in the closure of $X'\setminus X$.
Clearly $\cL(X)_{\gamma_0}=\cL(X')_{\gamma_0}$, so we can assume that
$X=X'$ is the subscheme of 
${\rm Spec}\, k[x_1,\ldots,x_n,y_1,\ldots,y_l]$ 
defined by equations $f_1=\ldots =f_l=0$ such that the arc $\gamma_0(t)
=(x^0(t),y^0(t))=(x^0_1(t),\ldots,x^0_n(t),y^0_1(t),\ldots,y^0_l(t))$ is
not contained in the subscheme of $X$ defined by 
$\det\frac{\partial f}{\partial y}=0$. Here
$\frac{\partial f}{\partial y}$ is the matrix of partial derivatives
$\frac{\partial f_i}{\partial y_j}$. 

Let $\gamma$ be an $A$-deformation
of $\gamma _0$ for some test-ring $A$, so $\gamma (t)=(x(t),y(t))$, where
$x(t)\in A[[t]]^n$, $y(t)\in A[[t]]^l$. Then by the Weierstrass division
theorem
$\det\frac{\partial f}{\partial y}(x(t),y(t))$ has a unique
representation as $q(t)u(t)$ where $u\in A[[t]]$ is invertible 
and $q$ is a monic polynomial whose reduction modulo the maximal ideal
$m\subset A$ is a power of $t$. Let $d$ denote the degree of $q$; it
depends only on $\gamma _0$, not on its deformation $\gamma$. We assume
that $d>0$ (otherwise we can eliminate~$y$). \emph{The idea of what follows is to consider $q$
as one of the unknowns.} 

More precisely, $A$-deformations of $\gamma_0$ 
are identified with solutions of the following system of equations. The
unknowns are $q(t)\in A[t]$, $x(t)\in A[[t]]^n$, and $y(t)\in A[[t]]^l$
such that $q$ is monic of degree $d$, $q(t)$ is congruent to $t^d$ modulo
$m$, and the reduction of $(x(t),y(t))$ modulo $m$ equals
$\gamma_0(t)=(x^0(t),y^0(t))$. The equations are as
follows:
\begin{equation}\label{1}
  \det\frac{\partial f}{\partial y}(x(t),y(t))\equiv 0 \mbox{ mod }q, 
\end{equation}
\begin{equation}\label{2}
   f(x(t),y(t))=0,
\end{equation}
where $f:=(f_1,\ldots,f_l)$. (Notice that if (\ref{1}) is satisfied then 
$q(t)^{-1}\det\frac{\partial f}{\partial y}(x(t),y(t))$ is
automatically invertible because it is invertible modulo $m$).

Now fix $r\ge 2$ and consider the following system of equations.
The unknowns are $q(t)\in A[t]$, $x(t)\in A[[t]]^n$, and 
$\bar y\in A[t]^l/(q^{r-1})$ such that $q$ is monic of degree $d$, $q(t)$ is
congruent to $t^d$ modulo $m$, the reduction of $x(t)$ modulo $m$ equals
$x^0(t)$, and the reduction of $\bar y$ modulo $m$ equals the reduction
of $y^0$ modulo $t^{r-1}$. The equations are as follows:
\begin{equation}\label{3}
\det\frac{\partial f}{\partial y}(x(t),\bar y)\equiv 0\mbox{ mod }q, 
\end{equation}
\begin{equation}\label{4}
f(x(t),\bar y)\in\im (q^{r-1}\frac{\partial f}{\partial y}(x(t),\bar y):
A[t]^l/qA[t]^l\to q^{r-1}A[t]^l/q^{r}A[t]^l).
\end{equation}
Condition (\ref{4}) makes sense because $f(x(t),\bar y)$ is well defined
modulo the image of $q^{r-1}\frac{\partial f}{\partial y}(x(t),\bar y)$. Notice that
(\ref{4})~is indeed an equation because it is equivalent to the condition
$\hat Cf(x(t),y(t))\equiv 0\mbox{ mod }q^{r}$, where $y(t)\in A[t]^l$ is
a preimage of $\bar y$ and $\hat C$ is the matrix adjugate to
$C:=\frac{\partial f}{\partial y}(x(t),y(t))$ 
(so $C\hat C=\hat CC=\det C$). This condition is equivalent to the
following equations, which do not involve a choice of $y(t)\in A[t]^l$
such that $y(t)\mapsto \bar y$:
\begin{equation}\label{5}
f(x(t),\bar y)\equiv 0\mbox{ mod }q^{r-1},
\end{equation}
\begin{equation}\label{6}
\hat B f(x(t) ,\bar y)\equiv 0\mbox{ mod }q^{r},
\end{equation}
where $B:=\frac{\partial f}{\partial y}(x(t),\bar y)$; notice that
(\ref{6}) makes sense as soon as (\ref{5}) holds.

\begin{lem}
The natural map from the set of solutions of (\ref{1}-\ref{2}) to the
set of solutions of (\ref{3}-\ref{4}) is bijective.
\end{lem}

\noindent
{\bf Proof.} Let $a$ be the minimal number such that $m^a=0$. We proceed
by induction on $a$, so we can assume that $a\ge 2$ and the lemma is
proved for $A/m^{a-1}$. Then there exists $\tilde y(t)\in A[t]^l$ such
that $\tilde y\mbox{ mod }q^{r-1}=\bar y$ and 
$f(x(t),\tilde y(t))\in m^{a-1}[t]^l$; such $\tilde y$ is unique modulo
$q^{r-1}A[t]^l\cap m^{a-1}[t]^l$. We have to find 
$z(t)\in q^{r-1}A[t]^l\cap m^{a-1}[t]^l$ such that
$f(x(t),\tilde y(t)-z(t))=0$, i.e., $Cz(t)=f(x(t),\tilde y(t))$, where 
$C:=\frac{\partial f}{\partial y}(x(t),\tilde y(t))$. (\ref{3}) implies
that $\det C=q(t)u(t)$ for some invertible $u\in A[t]$. So $z(t)$ is
unique. By (\ref{4}) 
$f(x(t),\tilde y(t))\in q^{r-1}CA[t]^l+q^{r}A[t]^l$. But 
$CA[t]^l\supset (\det C)A[t]^l=qA[t]^l$, so $f(x(t),\tilde y(t))=Cz(t)$
for some $z(t)\in q^{r-1}A[t]^l$. We have 
$Cz(t)=f(x(t),\tilde y(t))\equiv 0\mbox{ mod }m^{a-1}$, so 
$q(t)z(t)\equiv 0\mbox{ mod }m^{a-1}$ and finally
$q(t)\equiv 0\mbox{ mod }m^{a-1}$.
\hfill\qedsymbol

\medskip

So the set of $A$-deformations of $\gamma_0$ can be identified with the
set of solutions of the system (\ref{3}-\ref{4}). This system is
essentially finite because $x(t)$ is relevant only modulo $q^{r}$.
E.g., if $r=2$ we can write
$x(t)$ as $q^2(t)\xi (t)+\bar x$, $\xi\in A[[t]]^n$, $\bar x\in A[t]^n$,
$\deg\bar x<2d$, and consider $\xi (t)$, $\bar x$, $q(t)$, and $\bar y$
to be the unknowns (rather than $x(t)$, $q(t)$, $\bar y$); then
(\ref{3}-\ref{4}) becomes a finite system of equations for $q$, $\bar x$,
$\bar y$ (and $\xi$ is not involved in these equations). So
$\cL(X)_{\gamma_0}$ is isomorphic to $D^{\infty}\times Y_y$, where the
$k$-scheme $Y$ of finite type and the point $y\in Y(k)$ are defined as
follows: for every $k$-algebra $R$ the set $Y(R)$ consists of triples 
$(q,\bar x ,\bar y)$ where $q\in R[t]$ is monic of degree $d$, 
$\bar x\in R[t]^n/(q^2)$, $\bar y\in R[t]^l/(q)$,
$\det B\equiv 0\mbox{ mod }q$, 
$B:=\frac{\partial f}{\partial y}(\bar x,\bar y)$, 
$f(\bar x,\bar y)\equiv 0\mbox{ mod }q$, and 
$\hat B f(\bar x ,\bar y)\equiv 0\mbox{ mod }q^2$;
$y\in Y(k)$ corresponds to $q=t^d$,
$\bar x=x^0(t)\mbox{ mod }t^{2d}$, $\bar y=y^0(t)\mbox{ mod }t^{d}$.
\hfill\qedsymbol

\section{Rephrasing the proof from \S\ref{s:GK}}  \label{s:rephrasing}
\subsection{Deducing Theorem~\ref{main} from Proposition~\ref{p:BK}}
\subsubsection{The setting}   \label{sss:the setting}

Let $\BA^n$ denote the $n$-dimensional affine space over the field $k$.
Let $f$ be a morphism $\BA^{n+l}\to \BA^l$, i.e., 
$f=(f_1,\ldots ,f_l)$ and every $f_i$ is a polynomial
$f_i(x,y)$ where $x=(x_1,\ldots ,x_n)\in \BA^n$ and
$y=(y_1,\ldots ,y_l)\in \BA^l$. Set $Q:=\det ({\partial f\over\partial y})$. 

Set $X:=f^{-1}(0)\subset\BA^{n+l}$;  for a $k$-algebra $A$ we let $X(A)$ denote the set of $A$-points of $X$. Let $\Delta_X\subset X$ be the subscheme of zeros of $Q$. Let $N$ be a non-negative integer.

In \S\ref{ss:proofBK} we will prove the following

\begin{prop}   \label{p:BK}
(i) There exists a $k$-scheme $Z$ representing the following functor: for any $k$-algebra $R$, an $R$-point of $Z$ is a pair consisting of a monic polynomial $q\in R[t]$ of degree $N$ and an element of the set $\underset{r}{\underset{\longleftarrow}\lim} X(R[t]/(q^r))$ such that
the scheme-theoretic preimage of $\Delta_X$ in $\Spec R[t]/(q^2)$ equals 
$\Spec R[t]/(q)$.

(ii) $Z$ is a product of a $k$-scheme of finite type and a (typically infinite-dimensional) affine space.
\end{prop}

\begin{rem}
The property from statement (i) clearly implies that for \emph{every} $r\ge 2$ the scheme-theoretic preimage of $\Delta_X$ in $\Spec R[t]/(q^r)$ equals 
$\Spec R[t]/(q)$.
\end{rem}

\subsubsection{Deducing Theorem~\ref{main} from Proposition~\ref{p:BK}}
As explained at the beginning of \S\ref{ss:proofGK}, we can assume that the scheme $X$ from Theorem~\ref{main} equals $f^{-1}(0)$ for some morphism 
$f:\BA^{n+l}\to \BA^l$ and the image of the formal arc $\gamma_0 :\Spec k[[t]]\to X$ is not contained in $\Delta_X$. Then $\gamma_0^{-1} (\Delta_X)=\Spec k[t]/(t^N)$ for some $N\in\BN$. Let $Z$ be the scheme from Proposition~\ref{p:BK}(i) corresponding to this $N$. Let $q=t^N$, then the pair $(q,\gamma_0)$ defines a $k$-point of $Z$. The Weierstrass division theorem implies that the formal neighborhood of $(q,\gamma_0)$ in $Z$ is equal to the formal neighborhood of $\gamma_0$ in the scheme of formal arcs in $X$.  So Theorem~\ref{main} follows from Proposition~\ref{p:BK}(ii).

\subsection{Proof of Proposition~\ref{p:BK}}  \label{ss:proofBK}
\subsubsection{Representing $Z$ as a  limit}  \label{sss:Z as lim}
We will represent $Z$ as a projective limit of certain $k$-schemes $Z_r$ of finite type, $r\ge 2$, so that  for $r\ge 3$ the scheme $Z_{r+1}$ is isomorphic to a product of $Z_r$ and an affine space. Let us note that the schemes $Z_r$ defined below were secretly used in the proof of Theorem~\ref{main}. 

We define $Z_r$ to represent the functor that associates to a $k$-algebra $R$ the set of triples $(q, \bar{x}, \bar{y})$ satisfying the following conditions:

\begin{enumerate}
\item[(1)]
$q\in R[t]$ is a monic polynomial  of degree $N$;

\item[(2)]
$\bar{x}\in (R[t]/(q^r))^n$, $\bar{y}\in (R[t]/(q^{r-1}))^l$;

\item[(3)]
 $f(\bar{x},\bar{y})\equiv 0 \mbox{ mod } q^{r-1}$;
 
\item[(4)]
$Q(\bar x,\bar y)\equiv 0  \mbox{ mod } q$;

\item[(5)]
let $C:=\frac{\partial f}{\partial y}(\bar{x},\bar{y})$ and let $\hat{C}$ be the matrix adjugate to $C$, then
$$\hat{C} f(\bar x,\bar y)\equiv 0  \mbox{ mod } q^r;$$
using  (3), (4), and the equality $\hat{C} C=Q(\bar x,\bar y)$, one easily checks that the congruence modulo $q^r$ makes sense  (even though  $\bar{y}$ is defined only modulo~$q^{r-1}$);

\item[(6)]
if $r\ge 3$ then the element $q^{-1}Q(\bar x,\bar y)\in R[t]/(q^{r-2})$ is invertible.
\end{enumerate}
It is easy to check that this functor is indeed representable by an affine scheme of finite type over $k$. Now it remains to prove the following

\begin{lem}
(i) The canonical morphism $Z_{r+1}\to Z_r$ is smooth. Its fibers have dimension $nN$. 

(ii) If $r\ge 3$ then $Z_{r+1}$ is isomorphic (as a scheme over $Z_r$) to a product of $Z_r$ and an affine space.
\end{lem}

\begin{proof}
Let $V_r$ be the $k$-scheme whose $R$-points are pairs  $(q, \bar{x})$, where $q\in R[t]$ is a monic polynomial  of degree $N$ and 
$\bar{x}\in (R[t]/(q^r))^n$; of course, this scheme is isomorphic to an affine space of dimension $N(1+nr)$. For each $r$ we have a canonical morphism $Z_r\to V_r$ (forgetting $\bar y$). These morphisms are compatible with each other, so we get a morphism $\varphi_r:Z_{r+1}\to Z_r\underset{V_r}\times V_{r+1}$. To prove the lemma, one checks straightforwardly that $\varphi_r$ is etale, and for $r\ge 3$ the map  $\varphi_r$ is an isomorphism.\footnote{The map $\varphi_2$ is neither surjective nor injective, in general; both phenomena occur already if $n=l=N=1$ and $f(x,y)=y(y-P(x))$, where $P(x)$ is a polynomial.}

Let us only describe $\varphi_r^{-1}$ assuming that $r\ge 3$. An $R$-point of $Z_r\underset{V_r}\times V_{r+1}$ is a triple 
$$(q, \bar{x}, \bar{y}), \quad \bar{x}\in (R[t]/(q^{r+1}))^n,  \quad\bar{y}\in (R[t]/(q^{r-1}))^l$$
satisfying properties (3)-(6) from \S\ref{sss:Z as lim} (as usual, $q\in R[t]$ is monic  of degree $N$). It is easy to check that $\varphi_r^{-1}$ is given by Newton's formula
$$\varphi_r^{-1}(q, \bar{x}, \bar{y})=(q, \bar{x}, \tilde{y}-h), \quad h:=C^{-1}f(\bar x,\tilde y),
$$
where $\tilde y\in (R[t]/(q^r))^l$ is any preimage of $\bar{y}\in (R[t]/(q^{r-1}))^l$. Note that $h$ is a well-defined element of $(q^{r-1}R[t]^l)/(q^rR[t]^l)$: indeed,
$h=(q^{-1}Q(\bar x, \bar{y}))^{-1} q^{-1}\hat C f(\bar x,\tilde y)$.
\end{proof}

The rest of the article is devoted to the geometric interpretation of the schemes $Z_r$ in terms of the \emph{Newton groupoids}. In some sense, the idea goes back to Finkelberg and Mirkovi\'c (see \S\ref{ss:FM}).

\section{Introduction to the Newton groupoids}   \label{s:Intro Newton groupoid}
\subsection{The language of groupoids}
\subsubsection{Abstract groupoids}   \label{sss:abstract groupoids} 
Recall that an (abstract) groupoid is just a category in which all morphisms are invertible. So the data defining a groupoid are as follows: the set of objects $X$, the set of morphisms $\Gamma$, the ``source" map $p_1:\Gamma\to X$, the ``target" map $p_2:\Gamma\to X$, and the composition map $c:\Gamma\times_X\Gamma\to\Gamma$. These data should have certain properties; in particular, one should have the ``unit" map $e:X\to\Gamma$, $x\mapsto\id_x$; one should also have the inversion map $i:\Gamma\iso\Gamma$. (More details can be found in \cite[Expos\'e V]{SGA3} or \cite{CF}.)

\subsubsection{Groupoids in the category of $k$-schemes}   \label{sss:groupoids in schemes} 
In the situation of \S\ref{sss:abstract groupoids} one can consider $e$ and $i$ as a part of the data; then all the properties become \emph{identities} (i.e., they do not involve existence quantifiers). After this, the notion of groupoid in any category with fiber products becomes clear. In particular, one has the notion of groupoid in the category of $k$-schemes. More details can be found in \cite[Expos\'e V]{SGA3} or \cite{CF}.

One can also define the notion of groupoid in any category using the language of $S$-points, see  \cite[Expos\'e V]{SGA3}.

\subsubsection{Conventions}
In the situation of \S\ref{sss:abstract groupoids} or \S\ref{sss:groupoids in schemes} one says that $\Gamma$ is a groupoid on $X$ or that $\Gamma$ is a groupoid acting on $X$.

We usually write a groupoid as $\Gamma\underset{p_2}{\overset{p_1}\rightrightarrows} X$ or as 
$\Gamma\overset{(p_1,p_2)}\longrightarrow X\times X$ 
(without mentioning the composition map explicitlly). This will not lead to confusion because we are mostly interested in the situation where the map $\Gamma\overset{(p_1,p_2)}\longrightarrow X\times X$ is a birational isomorphism.

 \emph{From now on, we consider only groupoids in the category of $k$-schemes} (unless stated otherwise).
 
\subsubsection{Smooth groupoids and quotient stacks}
A groupoid $\Gamma\underset{p_2}{\overset{p_1}\rightrightarrows} X$ is said to be \emph{smooth} if $p_1$ is smooth. This is equivalent to $p_2$ being smooth (indeed, the inversion map $i:\Gamma\iso\Gamma$ interchanges $p_1$ and $p_2$).

Let $\Gamma\underset{p_2}{\overset{p_1}\rightrightarrows} X$ be a smooth groupoid such that the morphism $\Gamma\overset{(p_1,p_2)}\longrightarrow X\times X$ is quasi-compact and quasi-separated (this is a very mild assumption). Then one defines the \emph{quotient stack} $X/\Gamma$, which is an \emph{algebraic stack}, see \cite{LM, Ol, St}. ``Almost all" algebraic stacks can be obtained this way. (To get all of them, one has to allow $\Gamma$ to be an algebraic space rather than a scheme.)

\subsection{Pointy stacks and $\oMaps$}
\subsubsection{Pointy stacks}   \label{sss:Pointy stacks} 
By a \emph{pointy stack} we mean an algebraic $k$-stack locally of finite type which has a dense open substack isomorphic to the point $\Spec k$. Note that such an open substack is clearly unique.

If an action of an algebraic group $G$ on a $k$-scheme $X$ locally of finite type has a dense open orbit on which the action is free then the stack $X/G$ is pointy. E.g., $\BA^1/\BG_m$ is a pointy stack.

More generally, let $X$ be a $k$-scheme locally of finite type and $\Gamma$ a smooth groupoid acting on $X$ so that the corresponding morphism $\Gamma\to X\times X$ is an isomorphism over $U\times U$ for some dense open $U\subset X$. Then the stack $X/\Gamma$ is pointy.

\subsubsection{Maps from a curve to a pointy stack} \label{sss:oMaps}
Let $\cY\supset\Spec k$ be a pointy stack; let $\cY'\subset\cY$ be the reduced closed substack such that $\cY\setminus\cY'=\Spec k$. We assume that the diagonal morphism $\cY\to\cY\times\cY$ is separated. (The assumption is mild because usually the diagonal morphism is affine.)

On the other hand, let $C$ be a smooth curve over $k$ (e.g., $\BA^1$).

In this situation we define $\oMaps (C,\cY)$ to be the functor that associates to a $k$-scheme $S$ the set\footnote{A priori, such morphisms form a groupoid rather than a set. But separateness of  the diagonal morphism $\cY\to\cY\times\cY$ easily implies that this groupoid is a set.} of morphisms 
$f:C\times S\to\cY$ such that $f^{-1}(\cY')$ is finite over $S$.

\begin{conj}
In this situation the functor $\oMaps (C,\cY)$ is representable by an algebraic space locally of finite type over $k$.
\end{conj}

\subsubsection{Easy example} 
Let $\cY=\BA^1/\BG_m$. Then a morphism $C\times S\to\cY$ is just a line bundle on $C\times S$ equipped with a section. So $\oMaps (C,\cY)$ is the scheme parametrizing effective divisors on $C$; in other words, $\oMaps (C,\cY)$ is the disjoint union of the symmetric powers $\Sym^N C$, $N\ge 0$.

\subsubsection{Important example} \label{sss:Zastava}
Let $G$ be a reductive group, $B\subset G$ a Borel subgroup, and $U$ its unipotent radical. Let $\cY:=\overline{U\setminus G}/B$, where 
$\overline{U\setminus G}$ is the affine closure of $U\setminus G$ (i.e., the spectrum of the ring of regular functions on $U\setminus G$). Then $\oMaps (C,\cY)$ is known to be representable by a scheme locally of finite type over $k$, which is called the \emph{open Zastava scheme}, see \cite{FM}. (``Zastava" is the Croatian for ``flag".)

\subsubsection{Remark} \label{sss:Delta}
Let $\cY\supset\Spec k$ be a pointy substack with the following property: there exists an effective Cartier divisor $\Delta\subset\cY$ such that $\cY\setminus\Delta =\Spec k$. Moreover, let us fix such $\Delta$.  Then for any morphism $f:C\times S\to\cY$ as in \S\ref{sss:oMaps} the subscheme  $f^{-1} (\Delta )\subset C\times S$ is an $S$-family of effective divisors on $C$, so we get a morphism from $\oMaps (C,\cY)$ to the disjoint union of the symmetric powers $\Sym^N C$, $N\ge 0$.  The preimage of $\Sym^N C$ in $\oMaps (C,\cY)$ will be denoted by $\Maps^{\circ}_N (C,\cY)$.

\subsection{The goal}   \label{ss:thegoal}

\subsubsection{The setting}   \label{sss:2the setting}
Let $X$ and $Y$ be  separated $k$-schemes locally of finite type and $\varphi :X\to Y$ be a $k$-morphism. Let $U\subset X$ be the locus where $\varphi$ is etale.
We assume that $Y$ is smooth, $X$ is a locally complete intersection\footnote{We do \emph{not} assume that $X$ is a relative locally complete intersection with respect to $\varphi$.}, and $U$ is dense in $X$. 

In this situation one defines the \emph{different}; this is a canonical effective Cartier divisor $\Delta_X\subset X$ such that $X\setminus\Delta_X=U$. Namely, 
$\Delta_X$ is the divisor associated by Knudsen-Mumford \cite{KM} to the relative cotangent sheaf $\Omega^1_{X/Y}$ (note that in our situation 
$\Omega^1_{X/Y}$ has homological dimension 1 and vanishes on $U$, so the construction from \cite{KM} is applicable).

The reader may prefer to focus on the following particular case.

\subsubsection{Particular case}   \label{sss:particular}
Let $X\subset\BA^{n+l}$ be as in \S\ref{sss:the setting}, $Y=\BA^n$, and $\varphi :X\to \BA^n$ the projection. Assume that $\varphi :X\to\BA^n$ is etale on a dense open subset of $X$. This assumption means that the subscheme $\Delta_X\subset X$ from \S\ref{sss:the setting} is a Cartier divisor; this is the different.

\subsubsection{The goal}  \label{sss:thegoal}
In the situation of  \S\ref{sss:2the setting} we will construct in \S\ref{s:Defining Newton} for each $r\ge 2$ a smooth groupoid $\Gamma_r$ acting on
$X$, which is called  the $r$th \emph{Newton groupoid} of $\varphi:X\to Y$. It has the following properties:

(i) the morphism $\Gamma_r\to X\times X$ is an isomorphism over $(X\setminus\Delta_X)\times (X\setminus\Delta_X)$, so the stack 
$X/\Gamma_r$ is pointy in the sense of \S\ref{sss:Pointy stacks};

(ii) the action of $\Gamma_r$ becomes the identity\footnote{By definition, this means that  the morphisms $p_1,p_2:\Gamma_r\to X$ have equal restrictions to $p_1^{-1}(\Delta_X)$.} when restricted to $\Delta_X$, so the stack $X/\Gamma_r$ has the property from \S\ref{sss:Delta} with $\Delta:=\Delta_X/\Gamma_r$;

(iii) in the situation of \S\ref{sss:particular} the functor $\Maps^{\circ}_N (\BA^1,X/\Gamma_r )$ (see \S\ref{sss:Delta})  is representable by an open subscheme of the scheme $Z_r$ from \S\ref{sss:Z as lim}; if $r\ge 3$ the open subscheme equals $Z_r$, and if $r=2$ it equals $\im (Z_3\to Z_2)$.

More properties of $\Gamma_r$ will be formulated in \S\ref{sss:LieGamma_r} and \S\ref{ss:r and r+1}-\ref{ss:the group scheme}.

\subsection{Relation to \cite{FM}}  \label{ss:FM}
Finkelberg and Mirkovi\'c \cite{FM} proved Theorem \ref{main} in the particular case that $X=\overline{U\setminus G}$, where $G$ and $U$ are as in 
\S\ref{sss:Zastava}. They did it by considering  $\oMaps (\BA^1,\overline{U\setminus G}/B )$. The proof of Theorem \ref{main} given in \S\ref{s:GK} or 
\S\ref{s:rephrasing} secretly uses a similar strategy, with the groupoid $\Gamma_r$ playing the role of $B$; this is clear from \S\ref{sss:thegoal}(iii).

\subsection{The Lie algebroid of $\Gamma_r$}  \label{ss:LieGamma_r$}
\subsubsection{The notion of Lie algebroid}
Let $X$ be a scheme locally of finite type over $k$ and $\Theta_X$ its tangent sheaf.

Recall that a \emph{Lie algebroid} on $X$ is a sheaf $\mathfrak a$ on $X$ equipped with an $\cO_X$-module structure, a Lie ring structure, and an  \emph{anchor map} $\tau :\fa\to\Theta_X$, which is supposed to be both an  $\cO_X$-module morphism and a Lie morphism; moreover, if $f$ is a regular function on an open subset $U\subset X$ and $v_1,v_2\in H^0(U,\mathfrak a )$ then one should have
\[
[v_1,fv_2]=f[v_1,v_2]+((\tau (v_1))(f))\cdot v_2.
\]

For instance, $\Theta_X$ is a Lie algebroid with the anchor map being the identity.

We say that a Lie algebroid $\mathfrak a$  on $X$ is \emph{locally free} if $\mathfrak a$ is a locally free coherent $\cO_X$-module.

\subsubsection{The Lie algebroid of a smooth groupoid} \label{sss:algebroid of groupoid}
Let $\Gamma$ be a smooth groupoid on $X$ and let $e:X\to\Gamma$ be its unit. Define $\Lie(\Gamma )$ to be the normal bundle of $X=e(X)\subset\Gamma$. It is well known that $\Lie(\Gamma )$ has a natural structure of locally free Lie algebroid on $X$ (e.g., see \cite{CF}). In particular, the anchor map 
$\tau :\Lie(\Gamma )\to\Theta_X$ is just the map from the normal bundle of $e(X)\subset\Gamma$ to the normal bundle of $X_{diag}\subset X\times X$ induced by the morphism $\Gamma\overset{(p_1,p_2)}\longrightarrow X\times X$.

\subsubsection{An example of Lie algebroid} \label{sss:example of Lie algebroid}
In the situation of \S\ref{sss:2the setting}  set $$\fa_r:=(\varphi^*\Theta_Y) (-r\Delta_X).$$ It is easy to see that if $r\ge 1$ then $\fa_r\subset\Theta_X$ and moreover, $\fa_r$ is a Lie subalgebroid of $\Theta_X$. It is clear that this Lie algebroid is locally free, and the restriction of its anchor map to $X\setminus\Delta_X$ is an isomorphism (this is parallel to \S\ref{sss:thegoal}(i)). Moreover, if $r\ge 2$ then the image of the anchor map of $\fa_r$ is contained in $\Theta_X (-\Delta_X)$ (this is parallel to \S\ref{sss:thegoal}(ii)).

\subsubsection{A key property of $\Gamma_r$}  \label{sss:LieGamma_r}
The groupoid $\Gamma_r$ that we will construct has the following property: the anchor map induces an isomorphism $\Lie (\Gamma_r )\iso\fa_r$.

\subsection{Relation between $\Gamma_r$ and $\Gamma_{r+1}$ }   \label{ss:r and r+1}
By \S\ref{sss:LieGamma_r}, we have $$\Lie (\Gamma_{r+1})=(\Lie (\Gamma_r))(-\Delta_X).$$
It turns out that in a certain sense,
\begin{equation}   \label{e:r and r+1}
\Gamma_{r+1}=\Gamma_r(-\Delta_X).
\end{equation}
The precise meaning of \eqref{e:r and r+1} is as follows. First, one has a morphism of groupoids $\Gamma_{r+1}\to\Gamma_r$ inducing the identity on $X$ (which is the scheme of objects for both groupoids). Second, for any scheme $S$ flat over $X$, the map 
$$\Mor_X (S,\Gamma_{r+1})\to \Mor_X (S,\Gamma_r)$$
is injective, and an $X$-morphism $f:S\to\Gamma_r$ belongs to its image if and only if 
the restriction of $f$ to $S\times_X\Delta_X$ is equal to the composition $$S\times_X\Delta_X\to X\overset{e}\longrightarrow\Gamma_r\, ,$$
where $e:X\to\Gamma_r$ is the unit. 

\subsection{The restriction of $\Gamma_r$ to $\Delta_X$}   \label{ss:the group scheme}
By \S\ref{sss:thegoal}(ii), $\Gamma_r\underset{X}\times\Delta_X$ is a smooth group scheme over $\Delta_X$. Let us describe its fiber $(\Gamma_r)_z$ over a point $z\in \Delta_X$.

Let $n:=\dim_zX$. Let $m$ be the multiplicity of $z$ in $X\times_Yz$ (i.e., in the fiber of $\varphi$ corresponding to $z$); then $m\in\BN\cup\{\infty\}$, and since $z\in\Delta_X$ we have $m\ge 2$.
In \S\ref{ss:smoothness} we will show that

(i) if either $r\ge 3$ or $m\ge 3$ then $(\Gamma_r)_z\simeq\BG_a^n$;

(ii) if $m=r=2$, $z$ is a nonsingular point of $X$, and the characteristic of $k$ is not $2$ then $(\Gamma_r)_z\simeq\BG_m\ltimes\BG_a^{n-1}$, where $\lambda\in\BG_m$ acts on  $\BG_a^{n-1}$ as multiplication by $\lambda^2$;

(iii) if $m=r=2$, $z$ is a nonsingular point of $X$, and $k$ has characteristic 2 then $(\Gamma_r)_z\simeq \BG_a^n$;

(iv) if $m=r=2$ and $z$ is a singular point of $X$ then $(\Gamma_r)_z\simeq (\BZ/2\BZ)\times\BG_a^n$.

\section{Newton groupoids (details)}   \label{s:Defining Newton}

Let $\varphi :X\to Y$ be as in \S\ref{sss:2the setting}. Just as in \S\ref{sss:2the setting}, let $\Delta_X\subset X$ be the different and let $U=X\setminus\Delta_X$.
We are going to define the groupoids $\Gamma_r$ on $X$, $r\ge 2$, which were promised in \S\ref{sss:thegoal}. 

\subsection{$\Gamma_r$ as a scheme over $X\times X$}  \label{sss:separated case}
This scheme will be obtained from $X\times X$ by a kind of ``affine blow-up".

Note that $X\times_YX$ and the diagonal $X_{diag}$ are closed subschemes of $X\times X$ (because $X$ and $Y$ are separated). Let $I_1\subset\cO_{X\times X}$ be the sheaf of ideals of $X_{diag}$. Let $I_2\subset\cO_{X\times X}$ be the sheaf of ideals of $X\times_YX$. Then $\cO_{X\times X}\supset I_1\supset I_2$.

Let $D:=\Delta_X\times X$, $\tilde D:=X\times\Delta_X$; then $D$ and $\tilde D$ are effective Cartier divisors on $X\times X$ and 
$(X\times X)\setminus (D\cup \tilde D)=U\times U$. Let $j:U\times U\hookrightarrow X\times X$ be the open immersion.

We define $\cA_r\subset j_*\cO_{U\times U}$ to be the $\cO_{X\times X}$-subalgebra generated by $I_1((r-1)D)$, $I_2(rD)$, $I_1((r-1)\tilde D)$, $I_2(r\tilde D)$.

Finally, we set $\Gamma_r^{\varphi}=\Spec\cA_r$; this is a scheme affine over $X\times X$. Usually we write $\Gamma_r$ instead of $\Gamma_r^{\varphi}$.

\begin{lem}   \label{l:easy}
(i) The morphism $\Gamma_r\to X\times X$ is an isomorphism over $U\times U$. 

(ii) The morphisms $p_1,p_2:\Gamma_r\to X$ have equal restrictions to $p_1^{-1}(\Delta_X)$. Similarly, 
$p_1|_{p_2^{-1}(\Delta_X)}=p_2|_{p_2^{-1}(\Delta_X)}$.

(iii) Let $X'\subset X$ and $Y'\subset Y$ be open subschemes such that $\varphi (X')\subset Y'$. Let $\varphi':X'\to Y'$ be induced by $\varphi :X\to Y$. Then 
$\Gamma_r^{\varphi'}$ is obtained from $\Gamma_r^{\varphi}$ by base change $X'\times X'\to X\times X$.

(iv) $\Gamma_r^{\varphi}$ is not changed if $\varphi :X\to Y$ is composed with an etale morphism $Y\to\tilde Y$.
\end{lem}

\begin{proof}
Checking (i) and (iii) is straightforward. Statement (ii) follows from the inclusions $I_1\subset \cA_r(-D)$ and 
$I_1\subset \cA_r(-\tilde D)$.
To prove (iv), use (ii) and the fact that $X\times_YX$ and $X\times_{\tilde Y}X$ are equal in a neighborhood of $X_{diag}$ (because the morphism $Y\to\tilde Y$ is etale).
\end{proof}

\subsection{$\Gamma_r$ as a groupoid}
In \S~\ref{ss:flatness} we will prove that the two morphisms $\Gamma_r\to X$ are flat. Assuming this fact, we prove
\begin{prop}
There is a unique way to make $\Gamma_r\to X\times X$ into a groupoid on~$X$. 
\end{prop}

$\Gamma_r=\Gamma_r^\varphi$ is called  the $r$th \emph{Newton groupoid} of $\varphi :X\to Y$.

\begin{proof}
Because of Lemma~\ref{l:easy}(ii), we can assume that $X$ and $Y$ are affine.

By Lemma~\ref{l:easy}(i), we have an open embedding $U\times U\mono\Gamma_r$. Its image is sche\-ma\-tically dense in $\Gamma_r$ (this is clear from the definition of $\Gamma_r$). Moreover, using flatness of the two morphisms $\Gamma_r\to X$, we see that the open embeddings
\[
U\times U\times U\mono\Gamma_r\times_X\Gamma_r, \quad
U\times U\times U\times U\mono\Gamma_r\times_X\Gamma_r\times_X\Gamma_r
\]
also have schematically dense images. So there is at most one morphism 
\begin{equation}  \label{e:composition law}
\Gamma_r\times_X\Gamma_r\to\Gamma_r
\end{equation}
 over $X\times X$, and if it exists it automatically has the associativity property.

 Let us check that \eqref{e:composition law} exists. Recall that $\Gamma_r:=\Spec\cA_r$. The ideals $I_1,I_2\subset\cO_{X\times X}$ that were used in the definition of $\cA_r$ from  \S\ref{sss:separated case} have the following properties:
 \[
 p_{13}^*I_1\subset  p_{12}^*I_1+p_{23}^*I_1,\quad  p_{13}^*I_2\subset  p_{12}^*I_2+p_{23}^*I_2\, ,
 \]
 where $p_{12}$, $p_{13}$, $p_{23}$ are the three projections $X^3\to X^2$. 
 This implies that if $f$ is a  regular function on $\Gamma_r$ then its pullback with respect to the composed morphism
 \[
 U\times U\times U\overset{p_{13}}\longrightarrow U\times U\mono\Gamma_r
 \]
 extends to $\Gamma_r\times_X\Gamma_r$. This proves the existence of \eqref{e:composition law}.
 
 We also need the unit and the inversion map for $\Gamma_r$. 
 The automorphism of $X\times X$ that takes $(x_1,x_2)$ to $(x_2,x_1)$ clearly has a unique lift to an isomorphism $i:\Gamma_r\iso\Gamma_r$. 
 The diagonal embedding $X\to X\times X$ has a unique lift to a morphism $e:X\to\Gamma_r$ (to prove its existence, one checks that if $f$ is any regular function on $\Gamma_X$ then the restriction of $f$ to $U_{diag}\subset U\times U\subset\Gamma_r$ extends to $X_{diag}$). Then $e$ is the unit and $i$ is the inversion map for $\Gamma_r$ (it suffices to check the required identities on schematically dense open subschemes).
\end{proof}

\subsection{Relation between $\Gamma_r$ and $\Gamma_{r+1}$ }   \label{ss:again r and r+1}
In this subsection (which is not used in the rest of \S\ref{s:Defining Newton}) we verify the claim  of \S\ref{ss:r and r+1}.

\begin{prop}  \label{p:r+1 to r}
(i) There is a unique morphism of $(X\times X)$-schemes $\Gamma_{r+1}\to\Gamma_r$. Moreover, it is a morphism of groupoids.

(ii) The corresponding morphism 
$\Gamma_{r+1}\underset{X\times X}\times (\Delta_X\times\Delta_X) \to \Gamma_{r}\underset{X\times X}\times (\Delta_X\times\Delta_X)$ 
is trivial, i.e., it is equal to the composition 
$$\Gamma_{r+1}\underset{X\times X}\times (\Delta_X\times\Delta_X)\overset{p_i}\longrightarrow \Delta_X\overset{e}\longrightarrow \Gamma_{r}\underset{X\times X}\times (\Delta_X\times\Delta_X),$$
where $e$ is the unit section and $i$ equals $1$ or $2$.

(iii) The affine morphism $\Gamma_{r+1}\to\Gamma_r$  can be described as follows.\footnote{This description means that $\Gamma_{r+1}$ is obtained from $\Gamma_r$ by blowing up $e(\Delta_X)$ and then removing the strict transform of $\Delta_r$.} Define $\Delta_r\subset\Gamma_r$ to be the preimage of the divisor $\Delta_X$ with respect to $p_1:\Gamma_r\to X$ (or equivalently, with respect to $p_2:\Gamma_r\to X$). Let $\nu :U\times U\to\Gamma_r$ be the open immersion. Let $\cB\subset \nu_*\cO_{U\times U}$ be the quasi-coherent  unital $\cO_{\Gamma_r}$-algebra generated by $I(\Delta_r)$, where $I\subset\cO_{\Gamma_r}$ is the ideal of $e(\Delta_X )$. Then $\Gamma_{r+1}=\Spec\cB$.

(iv) for any scheme $S$ flat over $X$, the map 
\begin{equation}
\Mor_X (S,\Gamma_{r+1})\to \Mor_X (S,\Gamma_r)
\end{equation}
is injective, and an $X$-morphism $f:S\to\Gamma_r$ belongs to its image if and only if the restriction of $f$ to $S\times_X\Delta_X$ is equal to the composition $$S\times_X\Delta_X\to\Delta_X\overset{e}\longrightarrow\Gamma_r\, ,$$
where $e$ is the unit. 
\end{prop}

\begin{proof}
Statement (i) is clear: the morphism $\Gamma_{r+1}\to\Gamma_r$ comes from the obvious inclusion $\cA_r\subset\cA_{r+1}$. 

Statement (ii) essentially says that the composed map 
$$\cA_r\mono\cA_{r+1}\epi\cA_{r+1}\underset{\cO_{X\times X}}\otimes\cO_D$$
is equal to the composed map $\cA_r\overset{e^*}\longrightarrow\cO_{X_{diag}}\to\cA_{r+1}\underset{\cO_{X\times X}}\otimes\cO_D$. This is checked straightforwardly.

Let us prove (iii). Let $J\subset\cA_r$ be the ideal corresponding to the closed subscheme $e(\Delta_X)\subset\Gamma_r$. By Lemma~\ref{l:easy}(ii), 
$\cA_r(-D)=\cA_r(-\tilde D)$, so $\cA_r(D)=\cA_r(\tilde D)$ and $J(D)=J(\tilde D)$. Statement (iii) essentially says that
\begin{equation}  \label{e: paraphrase of iii}
\cA_{r+1}=\cC,
\end{equation}
where  $\cC\subset j_*\cO_{U\times U}$ is the $\cA_r$-subalgebra generated by $J(D)=J(\tilde D)$. Let us prove~\eqref{e: paraphrase of iii}. Since $J$ contains $I_1((r-1)D)$, $I_2(rD)$, $I_1((r-1)\tilde D)$, and $I_2(r\tilde D)$, we see that $\cA_{r+1}\subset\cC$. On the other hand, statement (ii) means that 
$J\cA_{r+1}=\cA_{r+1}(-D)$, so $J(D)\subset\cA_{r+1}$ and therefore $\cC\subset\cA_{r+1}$. This proves \eqref{e: paraphrase of iii} and statement (iii).

Statement (iv) follows from (iii).
\end{proof}

 \subsection{Flatness of $\Gamma_r$}    \label{ss:flatness}
 Let us prove that the two morphisms $\Gamma_r\to X$ are flat. (Later we will show that they are smooth, see Proposition~\ref{p:smoothness}). By Lemma~\ref{l:easy}, it is enough to consider the situation of \S\ref{sss:particular}. In this situation we will give a very explicit description of $\Gamma_r$ (see Lemma~\ref{l:theembedding} below).

We use the notation of $\S\ref{sss:the setting}$, and we assume that the projection $\varphi :X\to\BA^n=Y$ is generically etale (i.e., the restriction of $Q$ to $X$ is not a zero divisor).  Let $U\subset X$ be the locus $Q\ne 0$. 

Recall that $X\subset\BA^{n+l}$, and the coordinates in $\BA^{n+l}$ are denoted by $$x_1,\ldots ,x_n, y_1,\ldots y_l.$$ We have $X\times X\subset\BA^{n+l}\times\BA^{n+l}$. The coordinates in $\BA^{n+l}\times\BA^{n+l}$ will be denoted by $$x_i,y_j,\tilde x_i,\tilde y_j, \quad 1\le i\le n, \; 1\le j\le l.$$
By definition, $\Gamma_r:=\Spec A$, where $A\subset H^0(U\times U)$ is the subalgebra generated by all regular functions on $X\times X$ and also the following ones:
\begin{equation}   \label{e:xi}
\xi_i:=\frac{\tilde x_i-x_i}{Q(x,y)^r}, \quad \eta_j:=\frac{\tilde y_j-y_j}{Q(x,y)^{r-1}}\, ,
\end{equation}
\begin{equation}   \label{e:tilde xi}
\tilde\xi_i:=\frac{\tilde x_i-x_i}{Q(\tilde x,\tilde y)^r}, \quad \tilde\eta_j:=\frac{\tilde y_j-y_j}{Q(\tilde x,\tilde y)^{r-1}}\, .
\end{equation}

Let us now give an explicit description of the scheme $\Gamma_r$. Consider the morphism
\begin{equation}  \label{e:theembedding}
\Gamma_r\to X\times\BA^{n+l},
\end{equation}
where the map $\Gamma_r\to X$ is given by $x_i$'s and $y_j$'s, and the  map $\Gamma_r\to \BA^{n+l}$ is given by $\xi_i$'s and $\eta_j$'s.

\begin{lem}  \label{l:theembedding}
(i) The morphism \eqref{e:theembedding} identifies $\Gamma_r$ with the locally closed subscheme $\Gamma'_r\subset X\times\BA^{n+l}$ defined by the equation
\begin{equation}  \label{e:theequation}
\eta+\hat{C}(x,y)u(x,y,\xi ,\eta)=0 
\end{equation}
(which is a system of $l$ scalar equations) and the inequality
\begin{equation}  \label{e:inequality}
v(x,y,\xi ,\eta )\ne 0,
\end{equation}
where $C(x,y)$ is the matrix $\frac{\partial f}{\partial y}$, $\hat{C}(x,y)$ is the matrix adjugate to $C(x,y)$, and
\[
u(x,y,\xi ,\eta):=\frac{f(x+Q(x,y)^r\xi,y+Q(x,y)^{r-1}\eta )-f(x,y)-Q(x,y)^{r-1}C(x,y)\eta}{Q(x,y)^r},  
\]
\begin{equation}  \label{e:v}
v(x,y,\xi ,\eta ):=\frac{Q(x+Q(x,y)^r\xi,y+Q(x,y)^{r-1}\eta )}{Q(x,y)}\, .
\end{equation}  

(ii) The morphism $\Gamma_r\to X$ given by $x_i$'s and $y_j$'s is flat.
\end{lem}

Note that $u=(u_1,\ldots ,u_l)$ is a vector function. Also note that $v(x,y,\xi ,\eta )$ and $u_j(x,y,\xi ,\eta)$ are polynomials (not merely rational functions).

\begin{proof}
By \eqref{e:xi}, we have
\begin{equation}   \label{e:tilde x_i}
\tilde x_i=x_i+Q(x,y)^r\xi_i\, , \quad y_j=y_j+Q(x,y)^{r-1}\eta_j\, .
\end{equation}
So formula \eqref{e:v} says that 
\begin{equation}  \label{e:2v}
v(x,y,\xi ,\eta )=Q(\tilde x,\tilde y)/Q(x, y).
\end{equation}
Thus $Q(\tilde x,\tilde y)/Q(x, y)$ is a regular function on $\Gamma_r$. By symmetry, 
$Q(x,y)/Q(\tilde x,\tilde y)$ is also a regular function on $\Gamma_r$. Therefore the inequality \eqref{e:inequality} holds on $\Gamma_r$. It is easy to check that the equality \eqref{e:theequation}  also holds on $\Gamma_r$ (use the definition of $u$ from the formulation of the lemma and the equalities
$f(\tilde x,\tilde y)=f(x,y)=0$).

The coordinate ring of $\Gamma_r$ is generated by $x_i,y_j,\xi_i,\eta_j,v^{-1}$; this is clear from \eqref{e:tilde x_i} and the formulas 
$$\tilde\xi_i=v^{-r}\xi_i,\quad \tilde\eta_j=v^{1-r}\eta_j,$$
which follow from \eqref{e:xi}-\eqref{e:tilde xi} and  \eqref{e:2v}. So the morphism \eqref{e:theembedding} identifies $\Gamma_r$ with a closed subscheme of the locally closed subscheme $\Gamma'_r\subset X\times\BA^{n+l}$ defined by \eqref{e:theequation} and \eqref{e:inequality}. On the other hand, $\Gamma_r\supset U\times U$. So to prove statement (i), it remains to show that $U\times U$ is schematically dense in $\Gamma'_r$. This follows from flatness of the morphism $\Gamma'_r\to X$, which we are going to prove.

Note that \eqref{e:theequation} is a system of $l$ equations for a point in $X\times\BA^{n+l}$, so it suffices to check that the fiber of $\Gamma'_r$ over any point 
$(x_0,y_0)\in X$ has dimension $\le n$. Since $\Gamma'_r\times_XU=U\times U$, we can assume that $(x_0,y_0)\not\in U$, which means that $Q(x_0,y_0)=0$. Under this condition, we have to show that the set of solutions to the equation 
\begin{equation}   \label{e:our equation}
\eta+\hat{C}(x_0,y_0)u(x_0,y_0,\xi ,\eta)=0
\end{equation}  
(which is a system of $l$ equations for $n+l$ unknowns) has dimension $\le n$. To see this, note that $\det C(x_0,y_0)=Q(x_0,y_0)=0$, so $\hat C(x_0,y_0)$ has rank $\le 1$. Also note that $u(x_0,y_0,\xi,\eta)$ is a sum of a quadratic form in $\eta$  and a function of $\xi$ (this is clear from the definition of $u$ given in the formulation of the lemma). These two facts imply that for any $\xi$ there are at most two values of $\eta$ satisfying \eqref{e:our equation} 
\end{proof}

\subsubsection{The composition law in $\Gamma_r$}   \label{sss:composition law}
We have described $\Gamma_r$ as a subscheme of $X\times\BA^{n+l}$. In these terms, one can write an explicit formula for the composition map
\begin{equation}  \label{e:comp law}
\Gamma_r\times_X\Gamma_r\to \Gamma_r\, .
\end{equation}
A point of $\Gamma_r\times_X\Gamma_r$ is a collection $(x,y,\xi ,\eta ,\tilde\xi ,\tilde\eta )$, where $(x,y)\in X$, $\xi$ and $\eta$ satisfy conditions 
\eqref{e:theequation}-\eqref{e:inequality}, and $\tilde\xi$, $\tilde\eta$ satisfy similar conditions 
\[
\tilde\eta+\hat{C}(\tilde x,\tilde y)u(\tilde x,\tilde y,\tilde \xi ,\tilde \eta)=0, \quad v(\tilde x,\tilde y,\tilde \xi ,\tilde \eta )\ne 0;
\]
here $\tilde x,\tilde y$ are given by \eqref{e:tilde x_i}. It is straightforward to check that the map \eqref{e:comp law} is as follows:
\begin{equation}   \label{e:2composition law}
(x,y,\xi ,\eta ,\tilde\xi ,\tilde\eta )\mapsto (x,y, \xi+v(x,y,\xi ,\eta )^r\tilde\xi,\eta+v(x,y,\xi ,\eta )^{r-1}\tilde\eta),
\end{equation}
where $v(x,y,\xi ,\eta )$ is defined by \eqref{e:v}.

\subsection{Smoothness and the group schemes $(\Gamma_r)_z\,$, $z\in\Delta_X$}  \label{ss:smoothness}
\begin{prop}  \label{p:smoothness}
(i) The groupoid $\Gamma_r$ is smooth.

(ii) For $z\in\Delta_X$ the group scheme $(\Gamma_r)_z:=\Gamma_r\times_Xz$ is as described in \S\ref{ss:the group scheme}.
\end{prop}

\begin{proof}
By Lemma~\ref{l:theembedding}(ii), $\Gamma_r$ is flat over $X$. So it suffices to prove statement (ii).

Let $z\in\Delta_X$. Let $\Fib_z:=X\times_Yz$ be the corresponding fiber\footnote{If $z$ is a $k$-point we can write this fiber as $\varphi^{-1}(\varphi (z))$.} of $\varphi :X\to Y$. Since the question is local, we can assume that $X$ and $Y$ are as in \S\ref{ss:flatness}. Moreover, we can assume that $l$ (i.e., the number of the variables $y_j$) is equal to the dimension of the tangent space of $\Fib_z$ at $z$. This means that $C(x_0,y_0)=0$, where $C$ is the matrix 
$\frac{\partial f}{\partial y}$ and $(x_0,y_0)=z$.

By Lemma~\ref{l:theembedding} and formula~\eqref{e:2composition law}, $(\Gamma_r)_z$ is the subscheme of $\BA^{n+l}$ defined by the conditions
\begin{equation}   \label{e:1group}
\eta+\hat{C}(x_0,y_0)u(x_0,y_0,\xi ,\eta)=0, \quad v(x_0,y_0,\xi ,\eta )\ne 0
\end{equation}
and equipped with the group operation
\begin{equation}    \label{e:2group}
(\tilde\xi ,\tilde\eta )\cdot (\xi ,\eta ) = (\xi+v(x_0,y_0,\xi ,\eta )^r\tilde\xi,\eta+v(x_0,y_0,\xi ,\eta )^{r-1}\tilde\eta),
\end{equation}
where $u$ and $v$ are as in the formulation of Lemma~\ref{l:theembedding}. Since $C(x_0,y_0)=0$ we see that $\hat{C}(x_0,y_0)=0$ if $l>1$ and 
$\hat{C}(x_0,y_0)=1$ if $l=1$.

If $r\ge 3$ then it is easy to see that $u(x_0,y_0,\xi ,\eta)=0$, $v(x_0,y_0,\xi ,\eta)=1$, so $(\Gamma_r)_z\simeq\BG_a^n$.

Now suppose that $r=2$. If $l>1$ 
then $\hat{C}(x_0,y_0)=0$, and it is easy to check\footnote{Use formula \eqref{e:v} and note that $Q$ has zero differential at $(x_0, y_0)$ because $Q=\det C$, $C(x_0,y_0)=0$, and $l>1$.} that $v(x_0,y_0,\xi ,\eta)=1$, so $(\Gamma_r)_z\simeq\BG_a^n$. This agrees with \S\ref{ss:the group scheme}(i): indeed, $l$ is the dimension of the tangent space of $\Fib_z$ at $z$, so if $l>1$ then the multiplicity of $z$ in $\Fib_z$ is greater than $2$.

Now let $r=2$, $l=1$. Recall that $\frac{\partial f}{\partial y}$ vanishes at $(x_0,y_0)$, so the Taylor expansion of $f(x_0,y_0+\eta )$ looks as follows:
\[
f(x_0,y_0+\eta )=f(x_0,y_0)+a\eta^2+\ldots
\] 
It is easy to check that 
\[
u(x_0,y_0,\xi ,\eta)=a\eta^2+\frac{\partial f}{\partial x}(x_0,y_0)\cdot\xi ,\quad v(x_0,y_0,\xi ,\eta)=1+2a\eta ,
\]
so $(\Gamma_r)_z$ is the subscheme of $\BA^{n+1}$ defined by the conditions
\begin{equation} 
\eta+a\eta^2+ \frac{\partial f}{\partial x}(x_0,y_0)\cdot\xi=0 , \quad 1+2a\eta \ne 0
\end{equation}
and equipped with the group operation
\begin{equation}    
(\tilde\xi ,\tilde\eta )\cdot (\xi ,\eta ) = (\xi+(1+2a\eta)^2\tilde\xi,\eta+\tilde\eta+2a\eta\tilde\eta).
\end{equation}
If $a=0$ this group scheme is isomorphic to $\BG_a^n$, which agrees with \S\ref{ss:the group scheme}(i). If $a\ne 0$ the group scheme $(\Gamma_r)_z$ depends on whether  $\frac{\partial f}{\partial x}(x_0,y_0)=0$ (i.e., on whether $z$ is a singular point of $X$), and it is straightforward to check that $(\Gamma_r)_z$ 
is as described in~\S\ref{ss:the group scheme}(ii-iv).
\end{proof}

\subsection{Verifying the claim of \S\ref{sss:thegoal}(iii)}
For a $k$-algebra $A$, let $F(A)$ be the set of triples $(I,\bar x, \bar y)$, where $I\subset A$ is an ideal, $\bar x\in (A/I^r)^n$, $\bar y\in (A/I^{r-1})^l$. Consider the map
\begin{equation}  \label{e:XtoF}
X(A)\to F(A)
\end{equation}
that takes $(x,y)\in X(A)\subset A^n\times A^l$ to $(I,\bar x, \bar y)$, where $I$ is the ideal generated by $Q(x,y)$, $\bar x\in (A/I^r)^n$ is the image of $x\in A^n$, and $\bar y\in (A/I^{r-1})^l$ is the image of $y\in A^l$. It is easy to check that the map \eqref{e:XtoF} factors through the quotient set 
$X(A)/\Gamma_r (A)$. Since $F$ is an fppf sheaf, we get a map $$(X/\Gamma_r)(A)\to F(A).$$

Applying the above construction to $A=R[t]$, one gets a morphism 
\begin{equation}  \label{e:what we need}
\Maps^{\circ}_N (\BA^1,X/\Gamma_r )\to Z_r,
\end{equation}
where $\Maps^{\circ}_N (\BA^1,X/\Gamma_r )$ is as in \S\ref{sss:Delta} and $Z_r$ is as in \S\ref{sss:Z as lim}. Moreover, the morphism \eqref{e:what we need} factors through $Z'_r:=\im (Z_{r+1}\to Z_r)$ (recall that $Z'_r$ is an open subscheme of $Z_r$ and if $r\ge 3$ then $Z'_r=Z_r$). It is easy to check that the morphism \eqref{e:what we need} is a monomorphism.

It remains to prove that if $r\ge 3$ then for every triple $(q,\bar x,\bar y)\in Z_r(R)$ there exists a faithfully flat finitely presented $R[t]$-algebra $A$ such that the triple $(qA, \bar x, \bar y)\in F(A)$ belongs to the image of the map \eqref{e:XtoF}. Choose $x_0\in R[t]^n$, $y_0\in R[t]^l$ mapping to $\bar{x}\in (R[t]/(q^r))^n$, $\bar{y}\in (R[t]/(q^{r-1}))^l$.
It suffices to find a flat finitely presented morphism $\Spec A'\to\Spec R[t]$ with non-empty fibers over points of $\Spec R[t]/(q)$ and  an element $\eta\in (A')^l$ such that
\begin{equation}   \label{e:17}
f(x_0,y_0+q^{r-1}\eta)=0,
\end{equation}
\begin{equation}    \label{e:18}
q^{-1}Q(x_0,y_0+q^{r-1}\eta)\in (A')^\times .
\end{equation}
By \S\ref{sss:Z as lim}(6) and the assumption $r\ge 3$, the element $q^{-1}Q(x_0,y_0)$ is invertible modulo~$q$. So \eqref{e:17} is the only essential condition for $\eta$: one can always achieve \eqref{e:18} by modifying $A'$ slightly.

Let $C_0:= \frac{\partial f}{\partial y}  (x_0,y_0)$ and let $\hat C_0$ be the adjugate matrix, so $\hat C_0 C_0=Q(x_0,y_0)$. 
Write 
\[f(x_0,y_0+t\eta)=f(x_0,y_0)+tC_0\eta+t^2g(t,\eta), 
\] 
and then rewrite \eqref{e:17} as
\begin{equation}   \label{e:19}
q^{-1}Q(x_0,y_0)\eta+q^{-r}\hat C_0 f(x_0,y_0)+q^{r-2}\hat C_0 g(q^r,\eta)=0,
\end{equation}
Note that $q^{-1}Q(x_0,y_0)\in A$ and $q^{-r}\hat C_0 f(x_0,y_0)\in A^l$ by conditions (4)-(5) from \S\ref{sss:Z as lim}.

Let $W$ be the $A$-scheme whose $A'$-points are solutions to \eqref{e:19}. 
It suffices to show that the fiber of $W$ over any point of $\Spec A/(q)$ is non-empty and the morphism $W\to\Spec A$ is flat at each point of $W$ where $q$ vanishes. Since \eqref{e:19} is a system of $l$ equations for $l$ unknowns, it is enough to show that the fiber of $W$ over any point of $\Spec A/(q)$ is finite and non-empty. In fact, it has exactly one point because $q^{-1}Q(x_0,y_0)$ is invertible modulo $q$ and $r-2>0$.

\subsection{The Lie algebroid of $\Gamma_r\,$}
Let $r\ge 2$. As already mentioned in \S\ref{sss:example of Lie algebroid}, we have the Lie subalgebroid $\fa_r:=(\varphi^*\Theta_Y) (-r\Delta_X)\subset \Theta_X$. Using the definition of the anchor map $\tau :\Lie(\Gamma_r )\to\Theta_X$ given in \S\ref{sss:algebroid of groupoid}, one checks that $\Ker\tau=0$ and $\im\tau=\fa_r$ (e.g., one can use Lemma~\ref{l:theembedding}). Therefore $\tau$ induces an isomorphism of Lie algebroids $\Lie (\Gamma_r )\iso\fa_r$.

\end{document}